\documentclass[reqno,12pt]{amsart}
\usepackage{amsmath, latexsym, amsfonts, amssymb, amsthm, amscd}

\usepackage{color}
\newenvironment{DIFnomarkup}{}{} % see man latexdiff

\numberwithin{equation}{section}

\newtheorem{theorem}{Theorem}[section]
\newtheorem{lemma}[theorem]{Lemma}

\newtheorem{cor}[theorem]{Corollary}

\newtheorem{remark}[theorem]{Remark}

\newfont{\indic}{bbmss12}
\def\un#1{\hbox{{\indic 1}$_{#1}$}}

\newcommand{\eps}{\epsilon}

\newcommand{\gep}{\varepsilon}       % \ge already exists...

\newcommand{\E}{{\ensuremath{\mathbb E}} }

\renewcommand{\P}{{\ensuremath{\mathbb P}} }

\newcommand{\R}{{\ensuremath{\mathbb R}}}

\title[Hitting properties and uniqueness for self-similar SDE]{Hitting properties and non-uniqueness for SDE driven by stable processes}

\author{Julien Berestycki}
\thanks{}
\address{Laboratoire de Probabilit\'es et Mod\'eles Al\'eatoires Universit\'e Paris 6, 4 place Jussieu, 75252 Paris Cedex 05, France}
\email{$\{$julien.berestycki, leif.doering, lorenzo.zambotti$\}$@upmc.fr}

\author{Leif D\"oring}
\thanks{LD is supported by the Foundation Science Mat\'ematiques de Paris}
\author{Leonid Mytnik}
\thanks{LM is partly supported by the Israel Science Foundation}
\address{Faculty of Industrial Engineering and Management Technion Israel Institute of Technology, Haifa 32000, Israel}
\email{leonid@ie.technion.ac.il}

\author{Lorenzo Zambotti}
\thanks{JB, LM and LZ thanks the Isaac Newton Institute for Mathematical Sciences, Cambridge, for their invitation during which part of this work was produced.}

\begin{document}

\begin{abstract}
We study a class of self-similar jump type SDEs driven by H\"older-continuous drift and noise coefficients. Using the Lamperti transformation for positive self-similar Markov processes we obtain a necessary and sufficient condition for almost sure extinction in finite time. We then show  that pathwise uniqueness holds in a restricted sense, namely among solutions spending a Lebesgue-negligible amount of time at 0. A direct power transformation plays a key role.
\end{abstract}
\maketitle

\section{Introduction and Results}

In recent years there has been considerable interest in proving existence and especially uniqueness of solutions to SDEs driven by $\alpha$-stable L\'evy processes with H\"older continuous coefficients. In \cite{FL, LM, Fournier} pathwise uniqueness for equations with H\"older continuous noise coefficients was obtained in the spirit of the classical Yamada-Watanabe result for SDEs driven by Brownian motion. On the other hand when the noise is additive (i.e. the noise coefficient is constant) and it is the drift coefficient which is supposed to be H\"older, Priola \cite{Priola} extended the known results for SDEs driven by Brownian motion to SDEs driven by stable L\'evy processes. 

In the present work we want to focus on a family of SDEs which interpolate  between the two classes of problems described above since both the drift and the noise coefficients are chosen H\"older-continuous.
We study existence and uniqueness (or lack of uniqueness) of non-negative solutions $(Z_t)_{t\geq 0}$ to the stochastic differential equation (SDE) of jump type
	\begin{align}\label{2b}
		Z_t=Z_0 + \int_0^t Z_{s-} ^{\beta} \,dL_s+\theta \int_0^t Z_s^{\eta}\,ds, \qquad t\geq 0,
	\end{align}
where $\beta,\eta \in [0,1),\theta \ge 0$ and $(L_t)_{t\geq 0}$ is a spectrally positive $\alpha$-stable $\alpha\in (1,2)$  L\'evy process with Laplace exponent 
	\begin{align*}%\label{nu_alpha}
	\log \E[e^{-\lambda L_1}] = \lambda^\alpha = \int_{(0,\infty)} \left( e^{-\lambda x} - 1+\lambda x \right) \, \frac{\alpha(\alpha-1)}{\Gamma(2-\alpha)} \, x^{-1-\alpha} \, dx, \qquad \lambda\geq 0.
	\end{align*}
%\XXX{ do we need this : Equivalently, we can say that $L$ has 0 drift, no Brownian component and L\'evy measure $c_\alpha x^{-1-\alpha} \,  \un{]0,+\infty[}(x) \, dx$ with $c_\alpha=\frac{\alpha(\alpha-1)}{\Gamma(2-\alpha)}$.} 
Observe that the drift is non-locally Lipschitz precisely around a point where the noise coefficients is degenerate (i.e. equal to 0). Therefore, it is perhaps not so surprising that uniqueness might fail if solutions hit zero. One of our main results is that, in a certain regime, pathwise uniqueness indeed fails, and we can explicitly construct two different solutions (which are moreover both strong). 
\smallskip

It turns out that if we chose $\eta$ properly as a function of $\alpha$ and $\beta$ the solutions of (\ref{2b}) trapped in 0  will be self-similar. This, in fact, is a crucial ingredient for our analysis and we will henceforth assume 	
	\begin{equation}\label{ass}
	\theta\geq 0, \quad \alpha\in\, (1,2),\quad \beta\in [1-1/\alpha,1), \quad \eta=1-\alpha(1-\beta)\in[0,1).
	\end{equation}	

Although this choice of parameters might appear arbitrary at first sight, covers in fact many important special cases.
\begin{itemize}	
\item[-] Solutions to (\ref{2b}) with $\theta=0$ and $\beta=1/\alpha$ are called continuous state branching processes with stable branching mechanism. If $\theta>0$ and still $\beta=1/\alpha$, the additional drift can be interpreted as a state-dependent immigration to the system and was studied for more general immigration mechanisms in Chapter 10 of Li \cite{Li}. 
\item[-] In the forthcoming article Berestycki et al. \cite{BDMZ} the authors use a spatial version of the SDE (\ref{2b}) with $\beta=1/\alpha$ and $\eta = 2-\alpha$ to study generalized Fleming-Viot superprocesses with mutation. Since the problem seems to be of independent interest, questions of existence and uniqueness of solutions to (\ref{2b}) are studied here separately.

\item[-]  If $\beta=1-1/\alpha$, then $\eta=0$, and uniqueness follows  for any $\theta\geq 0$ by Li and Mytnik~\cite{LM}. 
\end{itemize}
The case $\beta=1$ is not covered by the range of parameters allowed by (\ref{ass}). Nevertheless, it is easily seen that then $\eta=1$ and for any parameter $\theta\geq 0$ the SDE (\ref{2b}) is a linear equation for which pathwise uniqueness is a simple consequence of the Lipschitz property of the coefficients.

\smallskip

	%with Laplace exponent 
	%\begin{align}\label{nu_alpha}
	%\log \E[e^{-\lambda L_1}] = \lambda^\alpha = \int_{(0,\infty)} \left( e^{-\lambda x} - 1+\lambda x \right) \, \frac{\alpha(\alpha-1)}{\Gamma(2-\alpha)} \, x^{-1-\alpha} \, dx, \qquad \lambda\geq 0.
	%\end{align}
%	Equivalently, we can say that $L$ has 0 drift, no Brownian component and L\'evy measure $c_\alpha x^{-1-\alpha} \,  \un{]0,+\infty[}(x) \, dx$ with $c_\alpha=\frac{\alpha(\alpha-1)}{\Gamma(2-\alpha)}$.

% 	\smallskip
%Equation \eqref{2b} lies in between the two classes of problems described above: the drift and noise coefficients are both H\"older-continuous, and the drift is non-locally Lipschitz precisely around a point where the noise coefficients is degenerate (i.e. equal to 0). Therefore, it is perhaps not so surprising that uniqueness might fail if solutions hit zero. One of our main results is that pathwise uniqueness indeed fails, and we can explicitly construct two different solutions (which are moreover both strong).
%\smallskip

Before stating the results let us fix some notation. We suppose that $(L_t)_{t\geq 0}$ is adapted to a stochastic basis $(\Omega,\mathcal{G},(\mathcal{G}_t)_{t\geq 0},\P)$ satisfying the usual conditions. A $(\mathcal G_t)_{t\geq 0}$-adapted stochastic process $(Z_t)_{t\geq 0}$ with almost surely {\it c\`adl\`ag} sample paths solving \eqref{2b} a.s. is said to be a {\it solution to Equation \eqref{2b}}. If a solution is adapted to  the augmented filtration of $(L_t)_{t\geq 0}$ then it is said to be a {\it strong solution to Equation \eqref{2b}}. We say that pathwise uniqueness holds for the SDE \eqref{2b} if for any two solutions $Z^1, Z^2$ defined on $\Omega$ we have $\P(Z^{1}_t = Z^{2}_t, \, \forall\;t\geq 0)=1$.
 \smallskip

 A first simple observation is that strong existence and pathwise uniqueness of non-negative solutions hold for the SDE (\ref{2b}) before
 $$T_0:=\inf\{t\geq 0: Z_t=0\},$$ 
 the first hitting time of $0$; indeed, the coefficients are Lipschitz continuous on $(\gep,\infty)$  for all $\gep>0$. In order to understand why uniqueness might fail, we first explain when the event $\{T_0<\infty\}$ has positive probability, since otherwise nothing needs to be proved. Using that solutions of Equation (\ref{2b}) are self-similar for the appropriate choice of $\beta$ and $\eta$, Lamperti's transformation --which will be recalled below-- can be applied to prove the following result:
\begin{theorem}\label{T0} 
Let $\alpha\in (1,2)$, and $\beta, \eta$ are chosen as in (\ref{ass}), then $T_0<\infty$ almost surely if and only if  $0\leq \theta< \Gamma(\alpha)$ and  $T_0=\infty$ almost surely if and only if $ \theta\geq \Gamma(\alpha)$.
\end{theorem}

To see how uniqueness might fail when (re)starting at zero let us suppose that  $Z$ is the unique solution up to $T_0$. If $\beta>1-1/\alpha$, one possible solution to the SDE (\ref{2b}) is always  the trivial solution
 $(\bar Z_t)_{t\geq 0} := (Z_{t\wedge T_0})_{t\geq 0}$, i.e. the solution trapped at zero after $T_0$. Hence, the existence of a non-trivial solution contradicts pathwise uniqueness in the classical sense. It does not if we weaken the set of possible 	solutions to the class $\mathcal S$ 
\begin{align*}
	\mathcal S:=\left\{(Z_t)_{t\geq 0}\,\Big \vert\,
	\text{ $Z\geq 0$ and} 
	\int_0^\infty  \un{\{Z_t=0\}}dt=0\quad \text{a.s.}\right\}
\end{align*}	 
which in particular rules out the trivial solution $\bar Z$. Of course it is not clear {\it a priori} whether there is a solution $Z\in\mathcal S$. Both the drift and the noise are null when solutions hit zero so that existence of strong solutions that leave zero is non-trivial.

\begin{theorem}\label{T1} Let $\alpha\in (1,2), \beta\in (1-1/\alpha,1)$ and suppose that $Z_0>0.$ Recall that $\eta =1-\alpha(1-\beta).$
\begin{itemize}
\item[A)] If $\theta>\frac{\Gamma(\alpha\beta)}{\Gamma(\eta)}$, then  there is a unique solution $Z\in \mathcal S$ to the SDE \eqref{2b}, which is moreover strong.
\item[B)] If $\theta\leq \frac{\Gamma(\alpha\beta)}{\Gamma(\eta)}$, then there is no solution $Z\in \mathcal S$ to the SDE \eqref{2b}.
\end{itemize}
\end{theorem}

\begin{remark}
Recently, Bass et al. have obtained a strong uniqueness result in a class similar to $\mathcal S$ but without the positivity requirement, for solutions of  a Brownian-driven SDE with H\"older diffusion coefficient and no drift (and a similar singularity at 0).
\end{remark}
	
To combine the two theorems notice that with the choice (\ref{ass}) of parameters $\alpha,\beta, \eta$ the inequality
\begin{equation}\label{<}
	\frac{\Gamma(\alpha\beta)}{\Gamma(\eta)} < \Gamma(\alpha),
\end{equation}
holds, see Lemma \ref{eta} below. Therefore, for $\beta\in  (1-1/\alpha,1)$, Theorems \ref{T0} and \ref{T1} define three regimes for the SDE (\ref{2b}) as $\theta$ varies. 
\begin{itemize}
\item[-] If $\theta \le {\Gamma(\alpha \beta) \over \Gamma(\eta)}$ then the set of solutions of type $\mathcal S$ is empty. We can not rule out the existence of solutions outside $\mathcal S$ other than $\bar Z$.
%and therefore the unique solution (\ref{2b}) is $(\bar Z_t)_{t\ge 0}$ which is absorbed in 0. 
\item[-] If $ {\Gamma(\alpha \beta) \over \Gamma(\eta)} <\theta <\Gamma(\alpha)$ there is a unique, strong, non-trivial solution in $\mathcal S$ which hits zero in finite time almost surely. When $\beta>1-1/\alpha, \bar Z$ is still a strong solution. Therefore, in this case,  we have a non-uniqueness phenomenon for solutions of Equation (\ref{2b}).
\item[-] Finally, if $\theta \ge \Gamma(\alpha)$ there is a unique solution which never hits zero.
\end{itemize}
It is interesting to note that the regimes can be equally obtained from the theory of positive self-similar Markov processes as we explain below in Section \ref{S: ss ext}.

\medskip
Let us finally discuss the connection to the particular boundary case $\alpha=2, \beta=1/\alpha=1/2$ (which is not covered by (\ref{ass})).  Equation (\ref{2b}) becomes, for $\theta\geq 0$ and $Z_0\geq 0$
\begin{align*}
Z_t=Z_0+\int_0^t\sqrt{Z_s}\,dB_s+\theta \,t, \qquad t\geq 0.
\end{align*}
We recognize in $(2Z_t, t\ge0)$ a squared-Bessel process of dimension $2\theta.$ The drift is constant, and pathwise uniqueness always holds due to the classical results of Yamada and Watanabe.
Since $\Gamma(0)=+\infty$, the interesting regime B in Theorem \ref{T1} reduces to the case $\theta=0$, where 0 is a trap for $Z$ by pathwise uniqueness. Since $\Gamma(2)=1$, the dichotomy of Theorem \ref{T0} corresponds to the fact that a Bessel process of dimension $2\theta$ hits 0 in finite time with positive probability iff $\theta<1$.

\subsubsection{Organization of the Proofs}	

In Section 2 we use the theory of positive self-similar Markov processes to prove Theorem \ref{T0}. The arguments for the proof of Theorem \ref{T1} are gathered in Section 3. Finally, in Section 4 we show how our results can be used to construct self-similar extensions of $(Z_{t\wedge T_0})_{t\geq 0}$.

\section{Self-similarity and the Proof of Theorem \ref{T0}}

A positive self-similar Markov process (pssMp) of index $\gamma$ is a strong Markov family $(\P^x)_{x>0}$ with coordinate process denoted by $Z$ in the Skorohod space of c\`adl\`ag functions  satisfying
\begin{align}\label{05}
\text{the law of } (cZ_{c^{-1/\gamma} t})_{t\geq 0}\text{ under }\P^x \text{ is given by }\P^{cx}
\end{align}
for all $c>0$. John Lamperti has shown in \cite{L} that this property is equivalent to the existence of a L\'evy process $\xi$ such that, under $\P^x$,  the process $( Z_{t\wedge T_0})_{t\geq 0}$ has the same law as
$\big( x \exp\big(\xi_{\tau({tx^{-1/\gamma})}} \big)\big)_{t\geq 0}$, where
\begin{align*}
\tau(t):=\inf\{s\geq 0: A_s> t\}
\qquad \text{and}\qquad A_t:=\int_0^t \exp\left(\frac{1}{\gamma}\xi_s\right) \, ds.
\end{align*}
Since this is all we need, we assume from now on that the L\'evy process $\xi$ is conservative, i.e. the lifetime is infinite. The proof of Theorem \ref{T0} is based on the equivalence 
\begin{align}\label{5}
T_0<\infty\quad \text{a.s. for all initial conditions }Z_0>0\quad \Longleftrightarrow \quad \xi\text{ drifts to }-\infty
\end{align}	
for pssMps which is due to Lamperti \cite{L}. In order to connect the SDE (\ref{2b}) to these results we start with a simple lemma.
\begin{lemma}\label{Lem1}
Suppose that $\beta\in [1-\frac{1}{\alpha},1)$, then, for any initial condition $x>0$, the SDE (\ref{2b}) admits a unique non-negative solution absorbed at zero. The induced Markov family $(\P^x)_{x>0}$ is self-similar of index $1/(1-\eta)\geq 1$.
\end{lemma}

\begin{proof}
Existence and pathwise uniqueness before hitting any level $\eps$ follows from the Lipschitz structure of the integrands in $(\eps,\infty)$. Sending $\eps$ to zero this carries over to solutions up to $T_0$.
To prove the self-similarity assertion, we abbreviate $\gamma=1/(1-\eta)$ to obtain
\begin{align*}		\begin{split}
cZ_{tc^{-1/\gamma}} &
=cZ_0+\int_0^{tc^{-1/\gamma}} cZ_{s-}^{\beta} dL_s + \theta\int_0^{tc^{-1/\gamma}} 
 cZ_s^{1-1/\gamma} ds
			\\ & = cZ_0+\int_0^{t} cZ_{(sc^{-1/\gamma})-}^{\beta} dL_{(c^{-1/\gamma}s)} + \theta\int_0^t c^{1-1/\gamma}Z_{sc^{-1/\gamma}}^{1-1/\gamma} ds
			\\ & = cZ_0+\int_0^{t} \left(cZ_{(sc^{-1/\gamma})-}\right)^{\beta} dL^c_s + 
 \theta\int_0^t \left(cZ_{sc^{-1/\gamma}}\right)^{1-1/\gamma} ds,\end{split}
		\end{align*}
		where the L\'evy process $L^c_t:=c^{1/(\alpha\gamma)}L_{tc^{-1/\gamma}}$ has the same distribution as $L$, and we have used in particular that $c^{1-1/(\alpha\gamma)}=c^{\beta}.$
		The self-similarity now follows from well-posedness of the SDE before hitting zero.
	\end{proof}
	
	Next, we calculate the L\'evy process $\xi$ corresponding to solutions of the SDE (\ref{2b}) via Lamperti's transformation. For $\theta=0$ and $\beta=1/\alpha$, i.e. the stable CSBPs without immigration, $\xi$ can be recovered from Proposition 2 of Kyprianou and Pardo \cite{KP} combined  with the generator calculations of Caballero and Chaumont \cite{CC2}.
	\begin{lemma}\label{Lem2}
		Suppose that  $\mathcal M$ is a Poisson point process on $(0,\infty)\times (0,\infty)$ with intensity measure $\mathcal M'(ds,dx)=ds\otimes c_\alpha e^x(e^x-1)^{-\alpha-1}\,dx$, then
		\begin{align}\label{levy}
			\xi_t:=\left(\theta+\int_0^\infty\left(\log (1+x)-x\right) c_\alpha x^{-1-\alpha}\,dx\right)t+\int_0^t\int_0^\infty x \,\mathcal{(M-M')}(ds,dx)
		\end{align}
		is the L\'evy process  corresponding under Lamperti's transformation to the pssMp $(\P^x)_{x>0}$ defined by the SDE (\ref{2b}).
	\end{lemma}
	\begin{proof}
		First note that $\xi$ can equivalently be written as
		\begin{equation}\label{9}\begin{split}
			\xi_t= & \left(\theta+\int_0^\infty\left(\log (1+x)-x\right) c_\alpha x^{-1-\alpha}\,dx\right)t
			\\ & +\int_0^t\int_0^\infty \log(1+x) \mathcal{(N-N')}(ds,dx),
			\end{split}
		\end{equation}
		where $\mathcal N$ is a Poisson point process on $(0,\infty)\times (0,\infty)$ with intensity measure $\mathcal N'(ds,dx)=ds\otimes c_\alpha x^{-1-\alpha}$. The equivalence follows for instance from Theorem II.1.8 of Jacod and Shiryaev \cite{JS} and the compensator calculation, for any measurable function $W$ with compact support,
		\begin{align*}
			\int_0^t\int_0^\infty W(s,x) \mathcal M'(ds,dx)
&=\int_0^t \int_0^\infty W(s,x)c_\alpha e^x(e^x-1)^{-1-\alpha}\,dx
			\\&=\int_0^t\int_0^\infty W(s,\log(x+1))c_\alpha x^{-1-\alpha}\,dx\\
			&= \int_0^t\int_0^\infty W(s,\log(x+1))\mathcal N'(ds,dx)
		\end{align*}
		so that the jump-measures of both Poissonian integrals have the same deterministic intensity. It\={o}'s formula (see page 44 of Ikeda and Watanabe \cite{IW}) applied to \eqref{9} directly shows that $M_t:=\exp(\xi_t)$ satisfies
		\begin{align}\label{linear}
			M_t = 1+\theta\int_0^t M_{s} \, ds + \int_0^t\int_0^\infty M_{s-}x\,(\mathcal{N-N'})(ds,dx).
		\end{align}
		Recalling from the L\'evy-It\={o} representation that
		\begin{align*}
			L_t:=\int_0^t\int_0^\infty x\, \mathcal{(N-N')}(ds,dx)
		\end{align*}
		is a spectrally positive $\alpha$-stable L\'evy process with Laplace exponent $\lambda^\alpha$ and inserting in (\ref{linear}) shows that $\exp(\xi_t)$ solves
		\begin{align*}
					M_t = 1+\theta\int_0^t M_{s} \, ds +\int_0^tM_{s-}\,dL_s.
		\end{align*}
		Next, we have to include the time-change: since $\gamma=1/(1-\eta)$, Lamperti's time-change becomes
		\begin{align*}
			 \tau(t) :=\inf\{s>0: A_s>t\},\quad A_t :=\int_0^t \lambda_s \, ds, 
			 \quad \lambda_s := \exp\{(1-\eta) \xi_s\}.
		\end{align*}
		If we set
		\begin{align*}
			\tilde L_t:= \int_0^{\tau(t)} \lambda_{s-}^{1/\alpha} \, dL_s, \qquad t>0,
		\end{align*}
then we claim that $(\tilde L_t)_{t\geq 0}$ has the same law as $(L_t)_{t\geq 0}$. Indeed, let us denote by 
$\mathcal{\tilde N}$, respectively $\mathcal{\tilde N}'$, the image measure of
$\mathcal{N}$, resp. $\mathcal{N}'$, under the map $(s,x)\mapsto (A_s,\lambda_{s-}^{1/\alpha}x)$. Then $\mathcal{\tilde N}$ is an optional random measure, whose compensator $\mathcal{\tilde N}'$ is equal to $\mathcal{N}'$, since using the change of variable $(A_s,\lambda_{s-}^{1/\alpha}x)=(r,y)$, we find
\[
\int_0^\infty\int_0^\infty f(A_s,\lambda_{s-}^{1/\alpha}x) \,ds c_\alpha x^{-1-\alpha} \, dx =\int_0^\infty\int_0^\infty f(r,y) \, \lambda_{\tau(r)}^{-1-\frac1\alpha+\frac1\alpha+1} \, dr \,c_\alpha  y^{-1-\alpha} dy.
\]
By \cite[Theorem II.1.8]{JS}, $\mathcal{\tilde N}$ and $\mathcal{N}$ have the same law. Therefore, since
\[
\begin{split}
\int_0^t\int_0^\infty x \, \mathcal{(\tilde N-\tilde N')}(ds,dx)
& = \int_0^\infty\int_0^\infty \un{(A_s\leq t)} \, \lambda_{s-}^{1/\alpha}\, x \, \mathcal{( N-N')}(ds,dx)
\\ & = \int_0^{\tau(t)} \lambda_{s-}^{1/\alpha} \, dL_s = \tilde L_t,
\end{split}
\]
the claim is proved. Plugging-in, we obtain
	\begin{align*}
		M_{\tau(t)} & =1+\theta\int_0^{\tau(t)} M_{s} \, ds +\int_0^{\tau(t)} M_{s-} dL_s\\
		&=1+\theta\int_0^t M_{\tau(u)} \, \dot\tau(u) \, du + \int_0^t M_{\tau(u)-}e^{-(1-\beta) \xi^{}_{\tau(u)-}}d\tilde L_u
		\\ & = 1+ \theta\int_0^t e^{\xi^{}_{\tau(u)}(1-(1-\eta))} \, du +  \int_0^t e^{\xi^{}_{\tau(u)-}(1-(1-\beta) )} d\tilde L_u
		\\ & = 1+\theta\int_0^t M_{\tau(u)}^{\eta} \, du +  \int_0^t M_{\tau(u)-}^{\beta} d\tilde L_u.
	\end{align*}
	Hence,  $(M_{\tau(t)})_{t\leq T_0}$ is a weak solution to the SDE (\ref{2b}) until first hitting zero. Uniqueness of solutions then implies that $\xi$ is the Lamperti transformed L\'evy process corresponding to the solution of the SDE (\ref{2b}).		
	\end{proof}

	\begin{cor}\label{C}
		Suppose that $\xi$ is the Lamperti transformed L\'evy process corresponding to the pssMp $(\P^x)_{x>0}$ induced by the solutions to the SDE (\ref{2b}), then
		\begin{align}\label{6}
			\E[\exp(\lambda\xi^{}_1)] = \exp\left(\lambda\left(\theta-\frac{\Gamma(\alpha-\lambda)}{\Gamma(1-\lambda)}\right)\right)
		\end{align}
		for $\lambda\in [0,1).$
	\end{cor}
	\begin{proof}
		All we need to do is to apply the exponential formula (see Theorem 25.17 of Sato \cite{Sat})) to the L\'evy process $\xi$ represented as in \eqref{9} 		\begin{align*}
			\E[\exp(\lambda\xi^{}_1)]
%			 &=\exp\Bigg(\lambda\left(\theta+\int_0^\infty\left(\log (1+x)-x\right) c_\alpha x^{-1-\alpha}\,dx\right)\\
%			&\quad +\int_0^\infty \big(e^{\lambda x}-1-\lambda x\big)c_\alpha e^x(e^x-1)^{-1-\alpha}\,\,dx\Bigg)\\
			&=\exp\Bigg(\lambda\left(\theta+\int_0^\infty\left(\log (1+x)-x\right) c_\alpha x^{-1-\alpha}\,dx\right)\\
			&\quad +\int_0^\infty \big((1+x)^\lambda-1-\lambda \log(1+x)\big)c_\alpha x^{-1-\alpha}\,\,dx\Bigg)\\
			&=\exp\Bigg(\lambda\theta +\int_0^\infty \big((1+x)^\lambda-1-\lambda x\big)c_\alpha x^{-1-\alpha}\,\,dx\Bigg).
		\end{align*}
		To calculate the inner integral we use twice partial integration to obtain
		\begin{align}\label{eq:7_1}
			{\int_0^\infty \big((1+x)^\lambda-1-\lambda x\big)c_\alpha x^{-1-\alpha}\,\,dx}
=\frac{\lambda(\lambda-1)}{\alpha(\alpha-1)}c_\alpha\int_1^\infty x^{\lambda-2} (x-1)^{-\alpha+1}\,dx.
		\end{align}
		Substituting $x$ by $1/y$ and recalling that $c_\alpha=\frac{\alpha(\alpha-1)}{\Gamma(2-\alpha)}$ then yields
		\begin{align}\label{eq:7_2}
		 {
			\frac{\lambda(\lambda-1)}{\Gamma(2-\alpha)}\int_0^1 \frac{1}{y^{\lambda-2}} \Big(\frac{1}{y}-1\Big)^{-\alpha+1}\frac{1}{y^2}\,dy}
			&=\frac{\lambda(\lambda-1)}{\Gamma(2-\alpha)}\int_0^1 x^{\alpha-\lambda-1} (1-x)^{1-\alpha}\,dx.
		\end{align}
		The integral can be reformulated via Beta-functions to obtain equality with
		\begin{align}\label{eq:7_3}
			\frac{\lambda(\lambda-1)}{\Gamma(2-\alpha)} B(\alpha-\lambda,2-\alpha)
			&=\frac{\lambda(\lambda-1)}{\Gamma(2-\alpha)} \frac{\Gamma(\alpha-\lambda)\Gamma(2-\alpha)}{\Gamma(2-\lambda)}	
=- \lambda\frac{\Gamma(\alpha-\lambda)}{\Gamma(1-\lambda)},
		\end{align}
		where for the first equality we used $B(x,y)=\Gamma(x)\Gamma(y)/\Gamma(x+y)$ and for the second $\Gamma(x+1)=x\Gamma(x)$.
	\end{proof}

	\begin{cor}\label{Lem4}Let $\beta\in[1-1/\alpha,1)$ and
		suppose that $\xi$ is the Lamperti transformed L\'evy process corresponding to the pssMp $(\P^x)_{x>0}$ induced by the solutions to the SDE (\ref{2b}), then
		\begin{itemize}
			\item[i)] $\xi^{}$ drifts to $-\infty$ if and only if $\theta<\Gamma(\alpha)$,
			\item[ii)] there is $0<a<1-\eta$ such that $\E[e^{a \xi_1^{}}]>1$ if and only if $\theta>\frac{\Gamma(\alpha\beta)}{\Gamma(\eta)}$.
		\end{itemize}
	\end{cor}
	
	\begin{proof}	
	Let us first recall that, by H\"older's inequality, the Laplace exponent $\psi(\lambda):=\log\E[e^{\lambda \xi_1^{}}]$ is convex whenever it is well-defined. Furthermore, it satisfies $\psi(0)=0$ and $\psi'(0+)=\E[\xi_1^{}]$. \\
	i) To verify the claim it suffices to check for which values $\theta$ the mean $\E[\xi_1^{}]$ is strictly negative which is equivalent to finding $\lambda>0$ such that $\psi(\lambda)<0$. By our explicit calculation we have $\psi(\lambda)=\lambda\left(\theta-\frac{\Gamma(\alpha-\lambda)}{\Gamma(1-\lambda)}\right), \lambda\in [0,1),$ so that $\xi^{}$ drifts to $-\infty$ if and only if there is $\lambda>0$ such that
		$\theta<\frac{\Gamma(\alpha-\lambda)}{\Gamma(1-\lambda)}.$
	Since the Gamma-function is continuous on $(0,\infty)$ and $\Gamma(1)=1$ this is possible if and only if $\theta<\Gamma(\alpha)$.\\
	ii) Let us first assume $\eta>0$. As the formula for the Laplace exponent is well-defined
  %(and extends continuously) 
for 
$\lambda=1-\eta$, the left-hand side of the claim is equivalent to
	\begin{align*}
		\theta>\frac{\Gamma(\alpha-(1-\eta))}{\Gamma(1-(1-\eta))}=\frac{\Gamma(\alpha\beta)}{\Gamma(\eta)}
	\end{align*}
	due to the convexity of $\psi$. 
Similarly in the case of $\eta=0$, we extend continuously the Laplace exponent to $\lambda=1$, and by taking 
 $\Gamma(0)=\infty$ (this we assume everywhere throughout the paper) we see that the left-hand side of the claim is
 equivalent to $\theta>0$ in this case. 
	\end{proof}
Next, we connect the regimes of the previous corollary, i.e. we verify \eqref{<} above:
\begin{lemma}\label{eta}
Suppose $ \alpha\in (1,2)$ and  $\beta\in [1-1/\alpha,1)$, then
\[
\frac{\Gamma(\alpha\beta)}{\Gamma(\eta)} < \Gamma(\alpha).
\]
\end{lemma}
\begin{proof}
Let $p,q,m,n>0$ with $(p-m)>0$ and $(q-n)<0$; then we claim that
\begin{equation}\label{strange}
\Gamma(p+n) \, \Gamma(q+m) > \Gamma(p+q) \, \Gamma(m+n).
\end{equation}
Let us first show that this implies the claim of the lemma. By the constraint on $\beta$, we can find $n$ such that $\alpha(1-\beta)<n<1$; since $\alpha>1$ we can set 
\[
p:=\alpha-n>0, \quad q:=n-\alpha(1-\beta)>0, \quad m:=1-n>0,
\]
and we obtain 
\[
p+n=\alpha, \quad q+m=1-\alpha+\alpha\beta=\eta, \quad p+q=\alpha\beta, \quad m+n=1.
\] 
Moreover $p-m=\alpha-1>0$ and $q-n=-\alpha(1-\beta)<0$, so that
the desired result is obtained.\\
Let us now prove \eqref{strange}. 
Define the maps  $f , g :\, (0,1)\to [0,+\infty)$  by
\[
f(x):=x^{p-m}, \qquad g(x):=(1 - x)^{q-n}.%, \quad h(x)= x^{m?1}(1 ? x)^{n?1}.
\]	
Let us consider a $B(m,n)$-variable $X$, with density
\[
\frac{\Gamma(m+n)}{\Gamma(m)\Gamma(n)} \, x^{m-1}(1 - x)^{n-1} \, \un{\{x\in\,(0,1)\}} \, dx.
\]
As $(p-m)>0$ and $(q-n)<0$, the maps $f$ and $g$ are monotone increasing on $[0, 1]$. Therefore, considering $(X,Y)$ i.i.d. we obtain
\[
0< \E( (f(X)-f(Y))(g(X)-g(Y))) = 2(\E(f(X)g(X)) - \E(f(X)) \, \E(g(X)))
\]
i.e.
\[
\E(f(X)g(X))> \E(f(X)) \, \E(g(X)).
\]
But this corresponds to
\[
\frac{\Gamma(m+n)}{\Gamma(m)\Gamma(n)} \, \frac{\Gamma(p)\Gamma(q)} {\Gamma(p+q)}
> \frac{\Gamma(m+n)}{\Gamma(m)\Gamma(n)} \, \frac{\Gamma(p)\Gamma(n)} {\Gamma(p+n)} \, \frac{\Gamma(m+n)}{\Gamma(m)\Gamma(n)} \, \frac{\Gamma(m)\Gamma(q)} {\Gamma(m+q)}
\]
and after cancellation we get \eqref{strange}. 
\end{proof}

With the preparation finished, our first theorem can be proved.
\begin{proof}[Proof of Theorem \ref{T0}]
	In Lemma \ref{Lem1} we proved that solutions to the SDE (\ref{2b}) absorbed at zero form a pssMps $(\P^x)_{x>0}$. The corresponding L\'evy process has been characterized in Lemma \ref{Lem2}. Combining the Equivalence (\ref{5}) with part i) of Corollary \ref{Lem4} 
	the claim follows.
\end{proof}

In fact, we calculated in Corollary \ref{Lem4} more than we needed for the proof of Theorem \ref{T0} since part ii) was not used. The equivalence will be used later in Section 4.

\section{Solutions after $T_0$ and the Proof of Theorem \ref{T1}}
	The proof of Theorem \ref{T1} is based on the simple power transformation $z\mapsto z^{1-\eta}$ which turns the H\"older continuous drift into a constant drift.
	
	\begin{lemma}\label{hinher}
			Suppose $\theta\geq 0$,  $\alpha\in (1,2), \beta\in(1-1/\alpha,1)$ and suppose that $(Z_t)_{t\geq 0}$ is a non-negative (strong) solution to the SDE (\ref{2b}) started at $Z_0>0$. Then $(Z^{1-\eta}_t)_{t\geq 0}$ is a non-negative (strong) solution to
			\begin{align}\label{3}\begin{split}
				V_t&= Z_0^{1-\eta} + (1-\eta)\left(\theta\,-\frac{\Gamma(\alpha\beta)}{\Gamma(\eta)}\right)\int_0^t  \un{\{V_s\neq 0\}}\,ds\\
				&\quad+ \int_0^t \int_{0}^\infty \left( \left(V_{s-}^{\frac{1}{(1-\eta)}}+V_{s-}^{\frac{\beta}{1-\eta}}\, x\right)^{1-\eta} - V_{s-}\right) (\mathcal{N-N'})(ds,dx),\quad t\geq 0,\end{split}
			\end{align}
			where $\mathcal N$ is a Poisson point process on $(0,\infty)\times (0,\infty)$ with intensity measure $\mathcal N'(ds,dx)=ds\otimes c_\alpha x^{-1-\alpha}$.
	\end{lemma}
		
		\begin{proof}
		 Let us first rewrite the SDE (\ref{2b}) via the L\'evy-It\={o} representation in the form
		\begin{align*}
			Z_t&=Z_0+\theta\int_0^t Z_s^{\eta}\,ds+\int_0^t\int_0^\infty Z_{s-}^{\beta} x\, (\mathcal{N-N'})(ds,dx),
		\end{align*}
		where $\mathcal N$ is the jump-measure of $L$ which has intensity $\mathcal N'(ds,dx)=ds\otimes c_\alpha x^{-1-\alpha}\,dx$. We cannot directly apply  It\={o}'s formula with $F(z)=z^{1-\eta}$ since $F$ is not smooth at the boundary of $[0,\infty)$ and cannot be extended to a concave function on $\R$. To surround this difficulty let us define $F_\eps(z)=(z+\eps)^{1-\eta}$, which is smooth for
		 $z\in [0,\infty)$, and 
		\[
		\begin{split}
			G(z,x,\eps)=F_\eps(z  +z^{\beta}\, x)-F_\eps(z)-F'_\eps(z)z^{\beta} x, \qquad z,x,\eps\geq 0.
			%=(\eps+z  +z^{1/\alpha}\, x)^{\alpha-1}-(\eps+z)^{\alpha-1}-(\alpha-1)(\eps+z)^{\alpha-1}z^{1/\alpha} x.
		\end{split}
		\]
		It\=o's formula then yields the almost sure identity
		\begin{align*}
		\begin{split}
			&\quad (Z_t+\eps)^{1-\eta} 
			- (Z_0+\eps)^{1-\eta} \\&= \theta\,(1-\eta)\int_0^t (Z_s+\eps)^{-\eta}Z_s^{\eta}\,ds\\
			&\quad+ \int_0^t \int_{0}^\infty \left( (Z_{s-} +\eps +Z_{s-}^{\beta}\, x)^{1-\eta} -( Z_{s-}+\eps)^{1-\eta} \right) (\mathcal{N-N'})(ds,dx)\\
			 &\quad + \int_0^t \int_{0}^\infty 
			 %\left( (Z_{s} +\eps +Z_{s}^{1/\alpha}\, x)^{\alpha-1} - (Z_{s}+\eps)^{\alpha-1}- 
			%(\alpha-1)\, (Z_{s}+\eps)^{\alpha-2}Z_s^{1/\alpha}\, x \, \right) ds \, 
			G(Z_s,x,\eps) \, 
			c_\alpha \, x^{-1-\alpha}\, dx\\
			&=: I^\eps_t+II^\eps_t+III^\eps_t.
			\end{split}
		\end{align*}
		In order to finish the proof we let $\eps$ tend to zero and show that the summands converge almost surely along a subsequence. It follows readily from dominated convergence that
		\begin{align*}
			\lim_{\eps\to 0} I_t^\eps=\theta(1-\eta)\int_0^t  \un{\{Z_s>0\}}\,ds.
		\end{align*}
Next, for $III^\eps_t$ we make the change of the variables $y=x\frac{Z_s^{\beta}}{Z_s+\eps}$ to get
\begin{align*}
III^\eps_t&= c_\alpha \int_0^t \left(\frac{Z_s}{Z_s+\eps}\right)^{\alpha\beta}\,ds\int_{0}^\infty 
  \left( (1+y)^{1-\eta} - 1- (1-\eta)y \right)y^{-1-\alpha}dy.
\end{align*}
Using (\ref{eq:7_1})-(\ref{eq:7_3}) with $\lambda=1-\eta$, one obtains 
\begin{align*}
III^\eps_t&= -(1-\eta)\frac{\Gamma(\alpha\beta)}{\Gamma(\eta)}\int_0^t \left(\frac{Z_s}{Z_s+\eps}\right)^{\alpha\beta}\,ds.
\end{align*}
Now, we can apply the dominated convergence theorem to obtain almost surely
		\begin{align*}
		\lim_{\eps\downarrow 0} III_t^\eps=  -(1-\eta)\frac{\Gamma(\alpha\beta)}{\Gamma(\eta)}\int_0^t 
              \un{\{Z_s>0\}}\,ds.
		\end{align*}
		 Next, we have to deal with the term $II_t^\eps$ for which we first show $L^p$-convergence for some $p\in(\alpha,2).$ Let us abbreviate
		\begin{align*}
			H(z,x,\eps)&:=F_\eps(z  +z^{\beta}\, x)-F_\eps(z) \geq 0
		\end{align*}
		satisfying
		\begin{align*}
			\frac{d}{d\eps}H(z,x,\eps)&=F'_\eps(z  +z^{\beta}\, x)-F'_\eps(z)\leq 0.
		\end{align*}
        Since $1-\eta=\alpha(1-\beta)<1$, we can fix $p\in (\alpha, \frac{1}{1-\beta}\wedge 2).$ 
		 Applying
                        Burkholder-Davis-Gundy inequality  (see e.g. \cite[Theorem VII.92]{DM83})
     we obtain
			\begin{equation}\begin{split}\label{K}
			&\quad\E\left[\left(\int_0^t \int_{0}^\infty \left( (Z_{s-}  +Z_{s-}^{\beta}\, x)^{1-\eta} -Z_{s-}^{1-\eta} \right) (\mathcal{N-N'})(ds,dx)\right.\right.\\
			&\quad\quad-\left.\left.\int_0^t \int_{0}^\infty \left( (Z_{s-} +\eps +Z_{s-}^{\beta}\, x)^{1-\eta} -( Z_{s-}+\eps)^{1-\eta} \right) (\mathcal{N-N'})(ds,dx)\right)^p\right]\\
			&\leq c_p \E\Bigg[\left(\int_0^t \int_{0}^\infty \big(H(Z_{s-},x,0)-H(Z_{s-},x,\eps)\big)^2
\mathcal{N}(ds,dx)\right)^{p/2}
%ds \,c_\alpha\, x^{-1-\alpha}dx
\Bigg]
\\
&\leq c_p
\E\Bigg[\int_0^t \int_{0}^\infty \big(H(Z_s,x,0)-H(Z_s,x,\eps)\big)^p
\mathcal{N}(ds,dx)\Bigg]
\\
&= c_p
\E\Bigg[\int_0^t \int_{0}^\infty \big(H(Z_s,x,0)-H(Z_s,x,\eps)\big)^p
\,c_\alpha\, x^{-1-\alpha}dx\,ds
\Bigg]
\end{split}
		\end{equation}
where $c_p>0$ is a constant coming from the Burkholder-Davis-Gundy inequality.
	%	if the right-hand side is finite. In order to verify the finiteness we use that
Since  $H$ is positive and pointwise decreasing in $\eps$, to show that the right-hand side of~(\ref{K})
  converges to zero, by monotone convergence theorem, it is enough to show the boundedness of 
		\begin{equation}\begin{split}
\label{eq:7_5}
\E\Bigg[\int_0^t \int_{0}^\infty H(Z_s,x,0)^p
\,c_\alpha\, x^{-1-\alpha}dx\,ds
\Bigg].
\end{split}
		\end{equation}
To this end, make 
 the change of variable  $x=Z^{1-\beta}_s y$ (note that the integrand is zero whenever $Z_s=0$) to obtain
		\begin{align*}
			&\E\Bigg[\int_0^t \int_{0}^\infty H(Z_s,x,0)^pds \,c_\alpha\, x^{-1-\alpha}\,dx\Bigg]\\
			&=\E\Bigg[\int_0^t \int_{0}^\infty Z_s^{(p-1)(1-\eta)}(\left(1 +y\right)^{1-\eta} -1)^p\,ds \,c_\alpha\, y^{-1-\alpha}dy\Bigg].
		\end{align*}
		To bound the right-hand side we use two bounds for the integrand. First, applying the H\"older property, gives
		\begin{align*}
			((1+y)^{1-\eta} -1^{1-\eta} )^p\leq y^{p(1-\eta)}
		\end{align*}
		and, secondly, we use the mean-value theorem with some $\zeta>0$ to obtain
		\begin{align*}
			((1+y)^{1-\eta} -1^{1-\eta})^p
			= (1-\eta)^p((1+\zeta)^{-\eta}y )^p
			\leq (1-\eta)^py^p.
		\end{align*}
		Plugging-in and using Fubini's theorem, we derive the upper bound
		\begin{align}\label{st}\begin{split}
			&\quad\E\Bigg[\int_0^t \int_{0}^\infty H(Z_s,x,0)^p\,ds \,c_\alpha\, x^{-1-\alpha}dx\Bigg]\\
			&\leq c_p \int_0^t\E\left[Z_s^{(p-1)(1-\eta)}\right]\,ds\, \int_{0}^\infty \min(y^p,y^{p(1-\eta)})\,c_\alpha\, y^{-1-\alpha}\,dy.\end{split}
		\end{align}
For the latter integral we estimate
\begin{align}
\nonumber
\int_{0}^\infty \min(y^p,y^{p(1-\eta)})\,c_\alpha\, y^{-1-\alpha}\,dy&\leq 
c_\alpha\int_{0}^1 y^{p-1-\alpha}\,dy+ c_\alpha\int_{1}^\infty  y^{p(1-\eta)-1-\alpha}\,dy\\
\nonumber
 &= c_\alpha\int_{0}^1 y^{p-1-\alpha}\,dy+ c_\alpha\int_{1}^\infty  y^{p\alpha(1-\beta)-1-\alpha}\,dy
\\
\label{eq:7_4}
\end{align}
which is finite since $p\in(\alpha, (1-\beta)^{-1}\wedge 2)$. 
		Defining $\tau_m=\inf\{t\geq 0: Z_t>m\}$, which tends to infinity almost surely by Lemma 2.3 of Fu and Li \cite{FL}, the defining equation (\ref{2b}) yields
		\begin{align*}
		\begin{split}
			\E[Z_{t\wedge \tau_m}] & =Z_0+\E\left[\theta \int_0^{t\wedge \tau_m} Z_s^{\eta}\,ds\right]\leq Z_0+ \theta \int_0^{t}(\E[Z_{s\wedge \tau_m}]+1)\,ds \\ & \leq  Z_0+ \theta \int_0^t\E[Z_{s\wedge \tau_m}]\,ds+\theta t.
		\end{split}
		\end{align*}
		Hence, by the Gronwall inequality and Fatou's lemma,
		\begin{align*}
		\begin{split}
			\E\big[Z_{t}^{(p-1)(1-\eta)}\big] & \leq \lim_{m\to \infty}\E\big[Z_{t\wedge \tau_m}^{(p-1)(1-\eta)}\big]\leq 1+ \lim_{m\to \infty}\E[Z_{t\wedge \tau_m}] \\ & \leq 1+\theta t+\theta^2\int_0^t s e^{\theta (t-s)}\,ds
		\end{split}
		\end{align*}
		so that  $ \int_0^t\E\big[Z_s^{(p-1)(1-\eta)}\big]\,ds<\infty$. This together with (\ref{eq:7_4})  implies that the right-hand side of (\ref{st}) is finite.\\ Now that the finiteness 
of~(\ref{eq:7_5})  is verified, (\ref{K}) and monotone convergence prove the 			convergence
		\begin{align*}
			II_t^\eps \stackrel{\eps\to 0}{\longrightarrow} \int_0^t \int_{0}^\infty \left( (Z_{s-}  +Z_{s-}^{\beta}\, x)^{1-\eta} -Z_{s-}^{1-\eta} \right) (\mathcal{N-N'})(ds,dx)
		\end{align*}
		in $L^p$.	Hence, there is a subsequence $\eps_k$ along which almost surely
		\begin{align*}
			\lim_{\eps_k\to 0}II_t^{\eps_k}=\int_0^t \int_{0}^\infty  \left( (Z_{s-}  +Z_{s-}^{\beta}\, x)^{1-\eta} -Z_{s-}^{1-\eta} \right) (\mathcal{N-N'})(ds,dx).
		\end{align*}
		Finally, along $\eps_k$ all summands $I_t^{\eps_k}$, $II_t^{\eps_k}$, $III_t^{\eps_k}$ converge almost surely so that we proved the semimartingale decomposition
			\begin{align*}
		\begin{split}
			&Z_t^{1-\eta} = Z_0^{1-\eta} + 
(1-\eta)\left(\theta-\frac{\Gamma(\alpha\beta)}{\Gamma(\eta)}\right)\int_0^t 
              \un{\{Z_s>0\}}\,ds
			\\ & + \int_0^t \int_{0}^\infty \left( (Z_{s-}+Z_{s-}^{\beta}\, x)^{1-\eta} - Z_{s-}^{1-\eta} \right) (\mathcal{N-N'})(ds,dx),
			\end{split}
		\end{align*}
so that, replacing $Z$ by $V^{1/(1-\eta)}$, the claim follows. 
		\end{proof}
		Here is the reverse power transformation with a small but crucial difference in the drift.
		
		\begin{lemma}\label{hinher2}
			Suppose $\theta\geq 0$,  $\alpha\in (1,2), \beta\in(1-1/\alpha,1)$, and suppose there exists a non-negative (strong) solution $(V_t)_{t\geq 0}$ to
			\begin{align}\label{3b}\begin{split}
				V_t&= V_0 +(1-\eta)\left(\theta-\frac{\Gamma(\alpha\beta)}{\Gamma(\eta)}\right)t\\
				&\quad+ \int_0^t \int_{0}^\infty \left( \left(V_{s-}^{\frac{1}{1-\eta}}+V_{s-}^{\frac{\beta}{1-\eta}}\, x\right)^{1-\eta} - V_{s-}\right) (\mathcal{N-N'})(ds,dx),\quad t\geq 0,\end{split}
			\end{align}
			started at $V_0>0$.
			Then $Z:=V^{\frac{1}{1-\eta}}$ is a non-negative (strong) solution of the SDE (\ref{2b}) with initial condition $V_0^{\frac{1}{1-\eta}}$.
	\end{lemma}	
		\begin{proof}
		Applying the Meyer-It\^o formula (see Theorem 51 of Protter \cite{P}) with the convex function $F(v)=z^{1/(1-\eta)}$, we obtain
		\begin{align}
\nonumber
			 Z_t&=V_t^{\frac{1}{1-\eta}}\\
\nonumber
			&=V_0^{\frac{1}{1-\eta}}+\left(\theta-\frac{\Gamma(\alpha\beta)}{\Gamma(\eta)}\right)\int_0^t V_s^{\frac{\eta}{1-\eta}}\,ds+\int_0^t\int_0^\infty V_{s-}^{\frac{\beta}{1-\eta}}x\,(\mathcal{N-N'})(ds,dx)\\
\nonumber
			&\quad + \int_0^t\int_0^\infty \left[V_{s}^{\frac{\beta}{1-\eta}}x- \frac{1}{1-\eta}V_{s}^{\frac{\eta}{1-\eta}}\left(\left(V_{s}^{\frac{1}{1-\eta}}+V_{s}^{\frac{\beta}{1-\eta}}\, x\right)^{1-\eta} - 			V_{s}\right)\right] \mathcal{N}'(ds,dx)\\
\label{eq8_1}
			&=Z_0+\left(\theta-\frac{\Gamma(\alpha\beta)}{\Gamma(\eta)}\right)\int_0^t Z_s^{\eta}\,ds+\int_0^t\int_0^\infty Z_{s-}^{\beta}x\,(\mathcal{N-N'})(ds,dx)\\
\nonumber
			&\quad -\frac{1}{1-\eta} \int_0^tZ_s^{\eta}\int_0^\infty \left[\left(Z_{s}+Z_{s}^{\beta}\, x\right)^{1-\eta}-Z_{s}^{1-\eta}-(1-\eta)Z_s^{\beta-\eta}x \right]
 c_\alpha x^{-1-\alpha}\,dx\,ds.
		\end{align}
		With the same integral identity used, in the proof of 
  Lemma~\ref{hinher} for analyzing $III^\eps,$ we get 
  that the inner integral in the last term on the right-hand side  equals to
\begin{align*}
\int_0^\infty G(Z_s,x,0)
 c_\alpha x^{-1-\alpha}\,dx=-(1-\eta)\frac{\Gamma(\alpha\beta)}{\Gamma(\eta)}.
\end{align*}
Substituting this into~(\ref{eq8_1}), 
we finally obtain 
		\begin{align*}
			Z_t&= Z_0+\theta\int_0^t Z_s^{\eta}\,ds+\int_0^t\int_0^\infty Z_{s-}^{\beta}x\,(\mathcal{N-N'})(ds,dx).
		\end{align*}	
		\end{proof}
		Before coming to the consequences of the power transformation we need existence and uniqueness for solutions of the jump type SDE (\ref{3b}). %The key point here is that the  condition (\ref{7}) for $\xi$ corresponds precisely to positive constant drift in (\ref{3b}).
		
	\begin{lemma}\label{exuniq}
		Suppose $V_0\geq 0$ and $\alpha\in (1,2), \beta\in(1-1/\alpha,1)$, then the two statements are equivalent:
		\begin{itemize}
		\item[i)] $\theta>\frac{\Gamma(\alpha\beta)}{\Gamma(\eta)}$,
		\item[ii)] there is a unique solution $V\in \mathcal S$ of the SDE (\ref{3b}) which is moreover strong.
		\end{itemize}
	\end{lemma}
		\begin{proof}
		\textbf{Step 1:} For the first part of the proof we assume $\theta>\frac{\Gamma(\alpha\beta)}{\Gamma(\eta)}$ and prove existence of a unique non-negative strong solution which we then show is of type $\mathcal S$.
		To ease notation, let us define
		\begin{align*}
			g(v,z)=v(1+v^{-1/\alpha}z)^{1-\eta}-v\quad \text{and}\quad c=(1-\eta)\left(\theta\,-\frac{\Gamma(\alpha\beta)}{\Gamma(\eta)}\right)
		\end{align*}
		so that the  (\ref{3b}) 
%(\ref{3})
 becomes 
		\begin{align*}
			V_t=V_0+ ct+\int_0^t\int_0^\infty g(V_{s-},x)(\mathcal{N-N'})(ds,dx).
		\end{align*}
		In the following we aim at applying the techniques developed in Li and Mytnik \cite{LM} though their Theorem 2.2 cannot be applied directly. Let us start with some estimates for $g$. Taking the derivative with respect to $v$ yields
		\begin{align*}
			{\partial g \over \partial v}(v,z)&=-\frac{1-\eta}{\alpha}\frac{z}{v^{1/\alpha}}\left( 1+ \frac{z}{v^{1/\alpha}}\right)^{-\eta} + \left(1+\frac{z}{v^{1/\alpha}}\right)^{1-\eta}-1\\
			&=-(1-\beta)x\left( 1+ x\right)^{-\eta} + (1+x)^{1-\eta}-1\\
			&=: f(x),
		\end{align*}
		where $x=\frac{z}{v^{1/\alpha}}$. It follows directly that $f(0)=0$ and furthermore
		\begin{align*}
			f'(x)&= (1-\beta)\eta x\left( 1+ x\right)^{-\eta-1} -
  (1-\beta)\left( 1+ x\right)^{-\eta}+ (1-\eta)(1+x)^{-\eta}\\
			&=(1-\beta)\left( 1+ x\right)^{-\eta-1}\left( (\eta x -( 1+ x) + \frac{1-\eta}{1-\beta}(1+x)\right)\\
			&= (1-\beta)\left( 1+ x\right)^{-\eta-1}(\alpha\beta x+\alpha-1)\\
			& >0
		\end{align*}
		for all $x>0$ whereas   the last equality follows as we recall the  definition of 
$\eta=1-\alpha(1-\beta)$, and the last inequality follows from the assumption  that $\alpha>1$. This implies that $f(x)$ is positive for all $x>0$, and hence $g(v,z)$ is increasing in $v$ for all $z$.  With this preparation we can find a modulus of continuity for $g$. Using the bound
		\begin{eqnarray}\label{eq1}
			(1+x)^{1-\eta}\leq 1+(1-\eta)x,\;x\geq 0,
		\end{eqnarray}
		 shows that
		 \begin{eqnarray*}
			0\leq f(x)\leq 1+(1-\eta)x -1 = (1-\eta)x. 
		\end{eqnarray*}
		Now assume without loss of generality that $v_1\leq v_2$. To estimate $|g(v_2,z)-g(v_1,z)|$ we consider two cases:

		\textbf{Case 1.} We first assume that $|v_2-v_1|\leq \frac{1}{2}v_2$. The previous calculations and the mean-value theorem yield (recall that $x= \frac{z}{v_1^{1/\alpha}}$),
		\begin{eqnarray*}
			|g(v_2,z)-g(v_1,z)| \leq  (1-\eta)x (v_2-v_1)  \leq  (1-\eta)\frac{z}{v_1^{1/\alpha}} (v_2-v_1),
		\end{eqnarray*}
		which combined with 
		\begin{eqnarray*} 
			v_1 \geq \frac{1}{2} v_2 \geq v_2-v_1
		\end{eqnarray*}
		gives the estimate
		\begin{align*}
			|g(v_2,z)-g(v_1,z)|   \leq  (1-\eta) z(v_2-v_1)^{1-1/\alpha}. 
		\end{align*}

		\textbf{Case 2}. Next we assume $|v_2-v_1|\geq \frac{1}{2}v_2$. In this case we will use the following (crude) bound (recall that again $v_2\geq v_1$ and $g(v,z)$ is increasing in $v$):
		\begin{align*}
			|g(v_2,z)-g(v_1,z)| 
			   \leq g(v_2,z)
			 &\leq v_2(1+  (1-\eta) v_2^{-1/\alpha} z -1)
			 \\ &=  (1-\eta)v_2^{1-1/\alpha}z
\\
&\leq c |v_2-v_1|^{1-1/\alpha} z, 
		\end{align*}
		where the second inequality follows from (\ref{eq1}) and the last inequality follows from the assumption of Case 2.
		\smallskip
		
		In total we obtain the following uniform modulus of continuity for $g$:
		 \begin{eqnarray}\label{eq2}
			|g(v_2,z)-g(v_1,z)|  \leq c |v_2-v_1|^{1-1/\alpha} z.
		\end{eqnarray}

		We are now in a position to prove existence for all $t\geq 0$ and pathwise uniqueness for (\ref{3}).
		\smallskip
		
		\textbf{Pathwise Uniqueness:}	The claim follows from Proposition 3.1 of Li and Mytnik \cite{LM} for which conditions i)-iii) are trivially matched and for 		condition iv) we apply their Lemma 3.2. To match with their notation, chose $\rho(v)=\rho_m(v)=v^{-1/2}, v\geq 0$, and $p=p_m=1-1/\alpha$. Then, by~(\ref{eq2}) the 
		condition (2c) from that paper is satisfied with $\rho$ and $p$ as above and with $f_m(z)=z$. Then by Lemma~3.2 from there we get that for any $h>0$ (note that $\mu_0(dz)= c_{\alpha}z^{-\alpha-1}dz$ there):
		\begin{align*}
			& \int_0^{\infty}D_{l_0(v_2,v_1,z)} \phi_k(v_2-v_1) c_{\alpha} z^{-\alpha-1}dz\\
			&\leq c_{\alpha} k^{-1} |v_2-v_1|^{2-2/\alpha -1}1_{|v_2-v_1|\leq a_{k-1}}\int_0^h z^{2-\alpha-1}\,dz
			 \\ &  + c_{\alpha} |v_2-v_1|^{1-1/\alpha}1_{|v_2-v_1|\leq a_{k-1}}\int_h^{\infty} z^{1-\alpha-1}\,dz\\\
			&=  c_{\alpha} k^{-1} |v_2-v_1|^{1-2/\alpha}1_{|v_2-v_1|\leq a_{k-1}} h^{2-\alpha}+ c_{\alpha} |v_2-v_1|^{1-1/\alpha}1_{|v_2-v_1|\leq a_{k-1}} h^{1-\alpha}.
		\end{align*}
		Now take $b_k = \ln(k)$ and $h=|v_2-v_1|^{1/\alpha} b_k$ and we get, 
		\begin{align*}
			& \int_0^{\infty}D_{l_0(v_2,v_1,z)} \phi_k(v_2-v_1) c_{\alpha} z^{-\alpha-1}dz\\
			&\leq c_{\alpha} k^{-1} |v_2-v_1|^{1-2/\alpha +2/\alpha-1} \ln(k)^{2-\alpha}  \un{|v_2-v_1|\leq a_{k-1}} 
			\\ &  + c_{\alpha} |v_2-v_1|^{1-1/\alpha+1/\alpha-1} \ln(k)^{1-\alpha}  \un{|v_2-v_1|\leq a_{k-1}}\\
			&\leq  c_{\alpha}  \un{|v_2-v_1|\leq a_{k-1}}(k^{-1} \ln(k)^{2-\alpha} +\ln(k)^{1-\alpha})
		\end{align*}
		which tends to zero as $k\rightarrow \infty$. Hence, condition iv) is satisfied and pathwise uniqueness follows from Proposition 3.1 of Li and Mytnik \cite{LM}.
		\smallskip
		
		\textbf{Strong Existence:} With the pathwise uniqueness in hands, strong solutions can now be constructed as in Section 5 of Fu and Li \cite{FL}; we only sketch the arguments. The condition $\theta>\frac{\Gamma(\alpha\beta)}{\Gamma(\eta)}$ enters here crucially since it assures a positive constant drift which pushes 		 solutions up whenever they hit zero. \\ First, one has to consider the truncated equations
		\begin{align*}
			V_t&=V_0+ ct+\int_0^t\int_\eps^mg(V_{s-},x)\wedge m\,(\mathcal{N-N'})(ds,dx),
		\end{align*}
		which have solutions $V^{\eps,m}$ due to Theorem 4.4 of Fu and Li \cite{FL}. It follows readily from Aldous' criterion that the sequence $V^{\eps,m}$ is tight for any $m$ fixed. Using the generators one can then verify weak convergence to a solution of 
		\begin{align*}
			V_t&=V_0+ ct+\int_0^t\int_0^mg(V_{s-},x)\wedge m\,(\mathcal{N-N'})(ds,dx).
		\end{align*}
		The pathwise uniqueness proof given above applies equally for this truncated version so that any subsequences $V^{\eps_k,m}$ converge to the unique strong solution $V^m$. The pathwise uniqueness then allows to get rid of the truncation $m$ 		as in the proof of Proposition 2.4 of Fu and Li \cite{FL}.
		\smallskip
		
		\textbf{Type $\mathcal S$:} Suppose $V$ is the unique strong solution of the SDE (\ref{3b}) constructed above. Then, by Lemma \ref{hinher2}, $Z:=V^{1/(1-\eta)}$ solves (\ref{2b}) and $Z^{1-\eta}$ solves the SDE (\ref{3}). Since $V=(V^{1/(1-\eta)})^{\eta-1}$ this shows 		that $V$ satisfies  (\ref{3}) and (\ref{3b}) so that equalizing both yields almost surely
		\begin{align*}
			&Z_0^{1-\eta} + (1-\eta)\left(\theta\,-\frac{\Gamma(\alpha\beta)}{\Gamma(\eta)}\right)\int_0^t  \un{\{V_s\neq 0\}}\,ds\\
			&\quad+ \int_0^t \int_{0}^\infty \left( \left(V_{s-}^{\frac{1}{1-\eta}}+V_{s-}^{\frac{\beta}{1-\eta}}\, x\right)^{1-\eta} - V_{s-}\right) (\mathcal{N-N'})(ds,dx)\\
			&=Z_0^{1-\eta} + (1-\eta)\left(\theta\,-\frac{\Gamma(\alpha\beta)}{\Gamma(\eta)}\right)t\\
			&\quad+ \int_0^t \int_{0}^\infty \left( \left(V_{s-}^{\frac{1}{1-\eta}}+V_{s-}^{\frac{\beta}{1-\eta}}\, x\right)^{1-\eta} - V_{s-}\right) (\mathcal{N-N'})(ds,dx).
		\end{align*}
		Hence, $\int_0^t  \un{\{Z_s\neq 0\}}\,ds=t$ almost surely which proves that $Z$ (and hence $V$) is of type $\mathcal S$.
		
		\textbf{Step 2:} Let us assume that $\theta\leq \frac{\Gamma(\alpha\beta)}{\Gamma(\eta)}$ and suppose there is a solution $V$ of type $\mathcal S$. By the power transformation of Lemma \ref{hinher2} the sequence of stopping times $\tau_m=\inf\big\{t\geq 0:V^{1/(1-\eta)}_t>m\big\}$ converges almost surely to infinity as argued in the proof of Lemma \ref{hinher}. In the case $\theta< \frac{\Gamma(\alpha\beta)}{\Gamma(\eta)}$, non-negativity combined with monotone convergence and Fatou's lemma implies that
%		and suppose there were a solution $V$ of type $\mathcal S$ to (\ref{3b}). Since $V$ is of type $\mathcal S$ we can use the power transformation in both directions and apply Lemma \ref{Lem1} combined with Corollary \ref{Lem4} and Equivalence (\ref{5}) to see that $V$ hits zero in finite time almost surely. Applying the strong Markov property to the almost surely finite stopping time $T_0$ we may assume without loss of generality that $V_0=0$. Further, we define the sequence of stopping times $\tau_m=\inf\big\{t\geq 0:Z_t=V^{1/(1-\eta)}_t>m\big\}$ which almost surely converges to infinity as argued in the proof of Lemma \ref{hinher}. By Fatou's lemma and monotone convergence we then obtain		
		\begin{align*}
				0&\leq\E[V_t] \leq \lim_{m\to\infty}\E[V_{t\wedge \tau_m}]\\
				&=V_0+ (1-\eta)\left(\theta\,-\frac{\Gamma(\alpha\beta)}{\Gamma(\eta)}\right)\lim_{m\to\infty}\E[t\wedge \tau_m]\\
				&= V_0+(1-\eta)\left(\theta\,-\frac{\Gamma(\alpha\beta)}{\Gamma(\eta)}\right)t.
		\end{align*}
		Hence, choosing $t$ large enough, we arrive at a contradiction since the right-hand side is negative. We conclude that there cannot be a solution of type $\mathcal S$.\\ 
		Next, suppose that $\theta=\frac{\Gamma(\alpha\beta)}{\Gamma(\eta)}$ which is strictly smaller than $\Gamma(\alpha)$ so that solutions almost surely hit zero in finite time. The proof of pathwise uniqueness given in Step 1 also applies in the case $\theta=\frac{\Gamma(\alpha\beta)}{\Gamma(\eta)}$. Now suppose there is a solution $V\in \mathcal S$. Consequently, since the constant drift is zero, also the stopped solution $(\bar V_t)_{t\geq 0}:=(V_{t\wedge T_0})_{t\geq 0}$ is a solution to (\ref{3b}). Since it is clearly different from $V$ as $\bar V\notin \mathcal S$, we found a contradiction to the pathwise uniqueness.
	\end{proof}
	We are now prepared to prove Theorem \ref{T1}.
	\begin{proof}[Proof of Theorem \ref{T1}]
		A) Let $(V_t)_{t\geq 0}$ the unique strong solution to (\ref{3b}) constructed in Lemma \ref{exuniq}. Then Lemma \ref{hinher2} shows that $Z=V^{\frac{1}{1-\eta}}$ is a strong solution to the SDE (\ref{2b}) which by Lemma \ref{exuniq} is of type $\mathcal S$. Now we suppose there are two solutions $Z^1, Z^2$ with $Z^1_0=Z^2_0$ that spend a Lebesgue null set of time at zero. Then the power transformation is reversible due to Lemmas \ref{hinher} and \ref{hinher2} so that the uniqueness follows from Lemma \ref{exuniq}.\\% In Lemma \ref{Lem1} we showed that (\ref{2b}) induces a pssMp so that  Corollary \ref{Lem4} combined with Equivalence (\ref{5}) shows that $T_0=\infty$ almost surely.\\		
		B) If there was a solution of class $\mathcal S$ to the SDE (\ref{2b}), taking the power $1-\eta$, there would be a solution of class $\mathcal S$ of the SDE (\ref{3b}). But this gives a contradiction to Lemma \ref{exuniq}.		
	\end{proof}

\section{Self-Similar Extensions }\label{S: ss ext}
	Lamperti's transformation for pssMps, which we recalled in Section 2, can not be used directly to characterize pssMps after hitting zero (or started from zero) since the infinite time-horizon $(0,\infty)$ for $\xi$ is compressed via the time-change to the possibly finite time-horizon $(0,T_0)$ so that 		the entire information on $\xi$ is already used until $T_0$. This drawback was resolved in recent years.
	\smallskip
	
	If $\xi$ drifts to $-\infty$, Rivero \cite{R2} and Fitzsimmons \cite{F} independently proved that a pssMp of index $\gamma$ has a recurrent self-similar Markovian extension of index $\gamma$ after $T_0$ with non-negative sample paths that leave zero continuously 
	 if and only if the following Cram\'er-type condition holds for the corresponding L\'evy process $\xi$:
	\begin{align}\label{7}
		\text{There is $0<a<\frac{1}{\gamma}$ such that }\E[e^{a \xi_1}]>1.
	\end{align}
%	If the Cram\'er-type condition fails, it was shown in Rivero \cite{R1} that self-similar extensions can be constructed by entering the positive half-line with a jump according to a power law and then restarting the process at the corresponding positive value (we  exclude \XXX{In fact do we exclude it ? Isn't it that it just doesn't solve our equation ?} this case here since it seems rather artificial in our context).
 Due to Corollary \ref{Lem4} ii), Condition (\ref{7}) is equivalent to $\theta>\frac{\Gamma(\alpha\beta)}{\Gamma(\eta)}$ which we found more directly to be the necessary and sufficient condition for the existence of non-trivial solutions to the SDE (\ref{2b}).
	
To further explore the connection between the present work and the results of Rivero and Fitzsimmons recall from Lemma \ref{Lem1} that solutions to the SDE (\ref{2b}) define a pssMp $(\P^x)_{x>0}$. If we define furthermore $(\bar \P^x)_{x\geq 0}$ via the  solutions $Z^x\in \mathcal S$ to the SDE (\ref{2b})  started at $x\geq 0$, then we can easily deduce the following consequence from the pathwise uniqueness:
	
	\begin{cor} Let $\beta\in [1-1/\alpha,1)$ and suppose $\theta>\frac{\Gamma(\alpha\beta)}{\Gamma(\eta)}$, then $(\bar \P^x)_{x\geq 0}$ is the unique extension of $(\P^x)_{x> 0}$ that leaves zero continuously.
		
	\end{cor}
	\begin{proof}
		It follows directly from the definition of $(\bar \P^x)_{x\geq 0}$ that it is an extension of $(\P^x)_{x>0}$ so that it suffices to prove the self-similarity for our solutions $Z^x\in \mathcal S$ to the SDE (\ref{2b}). Since, by construction, those are obtained by taking the power 
		$1-\eta$ of solutions to the SDE (\ref{3b}) it suffices to show that solutions to the SDE (\ref{3b}) are self-similar of index $1$.	
		 Setting $V^{c}_t := cV_{tc^{-1}}$ and
		plugging into the defining equation yields
		\begin{align*}
			&V^{c}_t  = cV_{tc^{-1}}= cV_0 + (1-\eta)\left(\theta\,-\frac{\Gamma(\alpha\beta)}{\Gamma(\eta)}\right)t\\
			&\quad+ \int_0^{t c^{-1}} \int_0^\infty \left( \left( \big(cV_{s-}\big)^{\frac{1}{1-\eta}}+\big(cV_{s-}\big)^{\frac{\beta}{1-\eta}}\, c^{\frac{1}{\alpha}} x\right)^{1-\eta} - cV_{s-}\right) (\mathcal{N-N'})(ds,dx)\\
			&= V_0^c + (1-\eta)\left(\theta\,-\frac{\Gamma(\alpha\beta)}{\Gamma(\eta)}\right)t\\
			&+ \int_0^{t } \int_0^\infty \left( \left( \big(V^{c}_{s-}\big)^{\frac{1}{1-\eta}}+\big(V^{c}_{s-}\big)^{\frac{\beta}{1-\eta}}\,  x\right)^{1-\eta} - V^{c}_{s-}\right) (\mathcal{N}_{(c)}-\mathcal{N}'_{(c)})(ds,dx),
		\end{align*}
		where $\mathcal{N}_{(c)}$ and $\mathcal{N}'_{(c)}$ are the image of $\mathcal{N}$, respectively $\mathcal{N}'$, under the map $(s,x)\mapsto(cs,c^{1/\alpha}x)$. Since $\mathcal{N}'_{(c)}=\mathcal{N}'$, $\mathcal{N}_{(c)}$ has the same law as $\mathcal{N}$, and we see that both $(V_t)_{t\geq 0}$ and 
		$(cV_{tc^-1})_{t\geq 0}$ are solutions to the same well-posed SDE so that they coincide in law. Solutions trivially leave zero continuously since the integrand is null at zero.		
	\end{proof}	
	If $\xi$ does not drift to $-\infty$, by (\ref{5}) almost surely the sample paths of the corresponding pssMp do not hit zero. The main question becomes whether the Markov family $(\P^x)_{x>0}$ can be extended continuously to $\P^0$.
	 Caballero and Chaumont \cite{CC} and later Chaumont et al. \cite{CKPR}
 proved that $\P^x$ converges weakly to a non-trivial limit law as $x$ tends to zero, i.e. it is a Feller process on $[0,\infty)$ and not on $(0,\infty)$ only, if and only if the overshoot process
 \begin{align*}
		\xi_{T_x}-x,\;x\geq 0,\quad\text{with }\quad T_x:=\inf\{t\geq 0:\xi_t\geq x\},
  \end{align*}
 converges, as $x\to\infty$, weakly towards the law of a finite random variable. A simpler construction of $\P^0$ has been given in Bertoin and Savov \cite{BS}
 via L\'evy processes indexed by the real line.\\ 
 In the case of the pssMps $(\P^x)_{x>0}$ corresponding to the SDE (\ref{2b}), the Feller property on $[0,\infty)$ is again a direct consequence of the uniqueness of Lemma \ref{exuniq}.
 \begin{cor} Let $\beta\in [1-1/\alpha,1)$ and
 	suppose that $\theta>\frac{\Gamma(\alpha\beta)}{\Gamma(\eta)}$, then $(\bar \P^x)_{x\geq 0}$ is weakly continuous in the initial condition.
 \end{cor}
\begin{proof}
	This follows directly from the uniqueness of Lemma \ref{exuniq} combined with  \cite[Theorem IX.4.8]{JS}.
\end{proof}
	
Our direct expression of the self-similar extension at zero is possible since the pssMp is given by a stochastic differential equation. In D\"oring and Barczy \cite{BD} this approach is extended by first reformulating Lamperti's transformation via jump type SDEs and then proceeding accordingly.

\end{document}